\documentclass{amsproc}
\usepackage{breqn,float,amssymb,pdflscape,rotating}
\makeatletter
\@namedef{subjclassname@2020}{%
  \textup{2020} Mathematics Subject Classification}
\makeatother
\newtheorem{theorem}{Theorem}[section]

\theoremstyle{definition}

\newtheorem{example}[theorem]{Example}

\theoremstyle{remark}

\numberwithin{equation}{section}

\begin{document}

\title[]{A Sextuple Integral Containing the Product of Associated Legendre polynomials $P_v^u(x) P_{\nu }^{\mu }(y)$: Derivation and Evaluation}


\author{Robert Reynolds}
\address[Robert Reynolds]{Department of Mathematics and Statistics, York University, Toronto, ON, Canada, M3J1P3}
\email[Corresponding author]{milver@my.yorku.ca}
\thanks{}

\author{ Allan Stauffer}
\address[Allan Stauffer]{Department of Mathematics and Statistics, York University, Toronto, ON, Canada, M3J1P3}
\email{stauffer@yorku.ca}
\thanks{This research is supported by NSERC Canada under Grant 504070}

\subjclass[2020]{Primary  30E20, 33-01, 33-03, 33-04, 33-33B, 33E20}

\keywords{Sextuple Integral, Cauchy integral, Associated Legendre polynomials}

\date{}

\dedicatory{}

\begin{abstract}
In this present paper we derive a six dimensional integral containing the product of the Associated Legendre Polynomials $P_v^u(x) P_{\nu }^{\mu }(y)$ where the indices are different and general. Included in the kernel of this integral is the generalized logarithmic function and coefficient logarithmic functions. The derivation of this integral is written in terms of the Hurwitz-Lerch zeta function and constant coefficients raised to a power. Special cases of this integral are derived in terms of fundamental constants and other special functions. All the results in this work are new.
\end{abstract}

\maketitle
\section{Significance Statement}
Definite integrals involving Associate Legendre Polynomials and products of these polynomials have been published in the works \cite{salam,samaddar,wong,bailey}. Diekema and Koornwinder \cite{diekema} investigated the integral for orthogonal polynomials over [-1,1] with respect to a general even orthogonality measure, with Gegenbauer and Hermite polynomials as explicit special cases. In this current paper we will extend the work done by Diekema and Koornwinder by increasing the dimensions of the integral and expanding the integrand by deriving a sextuple integral of a kernel involving the product of generalized Associate Legendre Polynomials $P_v^u(x) P_{\nu }^{\mu }(y)$  and a generalized logarithmic function raised to a power. This sextuple integral is derived in terms of the Hurwitz-Lerch zeta function. The indices of the Legendre polynomials are different and admit complex values. We will also look at orthogonal properties of this integral with respect to a weight function. An example of the definition of orthogonality is given by Askey, in equation (2.12) in \cite{askey}.
\section{Introduction}
In this paper we derive the sextuple definite integral given by
\begin{multline}
\int_{0}^{1}\int_{0}^{1}\int_{0}^{1}\int_{0}^{1}\int_{0}^{1}\int_{0}^{1}x^{m-1} y^{-m} \left(1-x^2\right)^{-u/2} \left(1-y^2\right)^{-\mu /2}
   P_v^u(x) P_{\nu }^{\mu }(y)\\
    \log ^{\frac{1}{2} (-\mu -m-\nu
   )}\left(\frac{1}{p}\right) \log ^{\frac{1}{2} (-\mu -m+\nu
   +1)}\left(\frac{1}{q}\right) \log ^{\frac{1}{2}
   (m-u+v)}\left(\frac{1}{t}\right) \log ^{\frac{1}{2}
   (m-u-v-1)}\left(\frac{1}{z}\right)\\
    \log ^k\left(\frac{a x \sqrt{\log
   \left(\frac{1}{t}\right)} \sqrt{\log \left(\frac{1}{z}\right)}}{y \sqrt{\log
   \left(\frac{1}{p}\right)} \sqrt{\log
   \left(\frac{1}{q}\right)}}\right)dxdydzdtdpdq
\end{multline}
where the parameters $k,a,$ are general complex numbers and $Re(u)<Re(m)<1/2, Re(m)<Re(v)$. This definite integral will be used to derive special cases in terms of special functions and fundamental constants. The derivations follow the method used by us in~\cite{reyn4}. This method involves using a form of the generalized Cauchy's integral formula given by
\begin{equation}\label{intro:cauchy}
\frac{y^k}{\Gamma(k+1)}=\frac{1}{2\pi i}\int_{C}\frac{e^{wy}}{w^{k+1}}dw.
\end{equation}
where $C$ is in general an open contour in the complex plane where the bilinear concomitant has the same value at the end points of the contour. We then multiply both sides by a function of $x$, $y$, $z$, $t$, $p$ and $q$ then take a definite quadruple integral of both sides. This yields a definite integral in terms of a contour integral. Then we multiply both sides of Equation~(\ref{intro:cauchy})  by another function of $y$ and take the infinite sum of both sides such that the contour integral of both equations are the same.
\section{Definite Integral of the Contour Integral}
We use the method in~\cite{reyn4,reyn5}. The variable of integration in the contour integral is $r =  w+m$. The cut and contour are in the first quadrant of the complex $r$-plane.  The cut approaches the origin from the interior of the first quadrant and the contour goes round the origin with zero radius and is on opposite sides of the cut.  Using a generalization of Cauchy's integral formula we form the sextuple integral by replacing $y$ by 
\begin{equation*}
\log \left(\frac{a x \sqrt{\log \left(\frac{1}{t}\right)} \sqrt{\log
   \left(\frac{1}{z}\right)}}{y \sqrt{\log \left(\frac{1}{p}\right)} \sqrt{\log
   \left(\frac{1}{q}\right)}}\right)
   \end{equation*}
   and multiplying by
   \begin{multline*}
    x^{m-1} y^{-m} \left(1-x^2\right)^{-u/2} \left(1-y^2\right)^{-\mu /2} P_v^u(x)
   P_{\nu }^{\mu }(y)\\
    \log ^{\frac{1}{2} (-\mu -m-\nu
   )}\left(\frac{1}{p}\right) \log ^{\frac{1}{2} (-\mu -m+\nu
   +1)}\left(\frac{1}{q}\right) \log ^{\frac{1}{2}
   (m-u+v)}\left(\frac{1}{t}\right) \log ^{\frac{1}{2}
   (m-u-v-1)}\left(\frac{1}{z}\right)
   \end{multline*}
   then taking the definite integral with respect to $x\in  [0,1] $, $y\in  [0,1] $, $z\in  [0,1] $ $t\in  [0,1] $ $p\in  [0,1] $  and $q\in  [0,1] $ to obtain
\begin{multline}\label{dici}
\frac{1}{\Gamma(k+1)}\int_{0}^{1}\int_{0}^{1}\int_{0}^{1}\int_{0}^{1}\int_{0}^{1}\int_{0}^{1}x^{m-1} y^{-m} \left(1-x^2\right)^{-u/2} \left(1-y^2\right)^{-\mu /2}
   P_v^u(x) P_{\nu }^{\mu }(y)\\
    \log ^{\frac{1}{2} (-\mu -m-\nu
   )}\left(\frac{1}{p}\right) \log ^{\frac{1}{2} (-\mu -m+\nu
   +1)}\left(\frac{1}{q}\right) \log ^{\frac{1}{2}
   (m-u+v)}\left(\frac{1}{t}\right) \log ^{\frac{1}{2}
   (m-u-v-1)}\left(\frac{1}{z}\right)\\
    \log ^k\left(\frac{a x \sqrt{\log
   \left(\frac{1}{t}\right)} \sqrt{\log \left(\frac{1}{z}\right)}}{y \sqrt{\log
   \left(\frac{1}{p}\right)} \sqrt{\log
   \left(\frac{1}{q}\right)}}\right)dxdydzdtdpdq\\\\
   =\frac{1}{2\pi i}\int_{0}^{1}\int_{0}^{1}\int_{0}^{1}\int_{0}^{1}\int_{0}^{1}\int_{0}^{1}\int_{C}a^w w^{-k-1} \left(1-x^2\right)^{-u/2} \left(1-y^2\right)^{-\mu /2}
   x^{m+w-1} y^{-m-w} P_v^u(x) P_{\nu }^{\mu }(y)\\
    \log ^{\frac{1}{2} (-\mu -m-\nu -w)}\left(\frac{1}{p}\right) \log ^{\frac{1}{2} (-\mu -m+\nu
   -w+1)}\left(\frac{1}{q}\right)\\
    \log ^{\frac{1}{2}
   (m-u+v+w)}\left(\frac{1}{t}\right) \log ^{\frac{1}{2}
   (m-u-v+w-1)}\left(\frac{1}{z}\right)dwdxdydzdtdpdq\\\\
   =\frac{1}{2\pi i}\int_{C}\int_{0}^{1}\int_{0}^{1}\int_{0}^{1}\int_{0}^{1}\int_{0}^{1}\int_{0}^{1}a^w w^{-k-1} \left(1-x^2\right)^{-u/2} \left(1-y^2\right)^{-\mu /2}
   x^{m+w-1} y^{-m-w} P_v^u(x) P_{\nu }^{\mu }(y)\\
    \log ^{\frac{1}{2} (-\mu -m-\nu -w)}\left(\frac{1}{p}\right) \log ^{\frac{1}{2} (-\mu -m+\nu
   -w+1)}\left(\frac{1}{q}\right)\\
    \log ^{\frac{1}{2}
   (m-u+v+w)}\left(\frac{1}{t}\right) \log ^{\frac{1}{2}
   (m-u-v+w-1)}\left(\frac{1}{z}\right)dxdydzdtdpdqdw\\\\
   =\frac{1}{2\pi i}\int_{C}\pi ^2 a^w w^{-k-1} 2^{\mu +u-1} \csc
   (\pi  (m+w))dw
\end{multline}
from equation (18.1.5) in \cite{bateman} and equation (4.215.1) in \cite{grad} where $Re(m+w)>0, Re(u)<1, 0<Re(m)<1, Re(v)>0, Re(m)<|Re(v)|, Re(\mu)<1, Re(\nu)>0, Re(m)<|Re(\nu)|$ and using the reflection formula (8.334.3) in \cite{grad} for the Gamma function. We are able to switch the order of integration over $x$, $y$, $z$, $t$, $p$ and $q$ using Fubini's theorem for multiple integrals see (9.112) in \cite{harrison}, since the integrand is of bounded measure over the space $\mathbb{C} \times [0,1] \times [0,1] \times [0,1] \times [0,1] \times [0,1] \times [0,1]$.
\section{The Hurwitz-Lerch Zeta Function and Infinite Sum of the Contour Integral}
In this section we use Equation~(\ref{intro:cauchy}) to derive the contour integral representations for the Hurwitz-Lerch Zeta  function.
\subsection{The Hurwitz-Lerch Zeta Function}
The Hurwitz-Lerch Zeta function (25.14) in \cite{dlmf} has a series representation given~by
\begin{equation}\label{lerch:eq}
\Phi(z,s,v)=\sum_{n=0}^{\infty}(v+n)^{-s}z^{n}
\end{equation}
where $|z|<1, v \neq 0,-1,..$ and is continued analytically by its integral representation given~by
\begin{equation}\label{armenia:eq8}
\Phi(z,s,v)=\frac{1}{\Gamma(s)}\int_{0}^{\infty}\frac{t^{s-1}e^{-vt}}{1-ze^{-t}}dt=\frac{1}{\Gamma(s)}\int_{0}^{\infty}\frac{t^{s-1}e^{-(v-1)t}}{e^{t}-z}dt
\end{equation}
where $Re(v)>0$, and either $|z| \leq 1, z \neq 1, Re(s)>0$, or $z=1, Re(s)>1$. 
\subsection{Infinite sum of the Contour Integral}
Using equation (\ref{intro:cauchy}) and replacing $y$ by 
\begin{equation*}
\log (a)+i \pi  (2 y+1)
\end{equation*}
 then multiplying both sides by
\begin{equation*}
 -i \pi ^2 e^{i \pi  m (2 y+1)} 2^{\mu +u}
 \end{equation*}
  taking the infinite sum over $y\in[0,\infty)$ and simplifying in terms of the Hurwitz-Lerch Zeta function we obtain
\begin{dmath}\label{isci}
\frac{i^{k-1} \pi ^{k+2} e^{i \pi  m} 2^{k+\mu +u} \Phi \left(e^{2 i m
   \pi },-k,\frac{\pi -i \log (a)}{2 \pi }\right)}{\Gamma(k+1)}\\\\
   =-\frac{1}{2\pi i}\sum_{y=0}^{\infty}\int_{C}i \pi ^2 a^w w^{-k-1} 2^{\mu +u} e^{i \pi  (2 y+1)
   (m+w)}dw\\\\
    =-\frac{1}{2\pi i}\int_{C}\sum_{y=0}^{\infty}i \pi ^2 a^w w^{-k-1} 2^{\mu +u} e^{i \pi  (2 y+1)
   (m+w)}dw\\\\
   =\frac{1}{2\pi i}\int_{C}\pi ^2 a^w w^{-k-1}2^{\mu +u-1} \csc (\pi  (m+w))dw
\end{dmath}
from equation (1.232.3) in \cite{grad} where $Im(\pi  (m+w))>0$ in order for the sum to converge.
\section{Definite Integral in terms of the Hurwitz-Lerch Zeta Function}
\begin{theorem}
For all $k,a \in\mathbb{C}, Re(u)<1, 0<Re(m)<1, Re(v)>0, Re(m)<|Re(v)|, Re(\mu)<1, Re(\nu)>0, Re(m)<|Re(\nu)|$ then,
\begin{multline}\label{dilf}
\int_{0}^{1}\int_{0}^{1}\int_{0}^{1}\int_{0}^{1}\int_{0}^{1}\int_{0}^{1}x^{m-1} y^{-m} \left(1-x^2\right)^{-u/2} \left(1-y^2\right)^{-\mu /2}
   P_v^u(x) P_{\nu }^{\mu }(y)\\
    \log ^{\frac{1}{2} (-\mu -m-\nu
   )}\left(\frac{1}{p}\right) \log ^{\frac{1}{2} (-\mu -m+\nu
   +1)}\left(\frac{1}{q}\right) \log ^{\frac{1}{2}
   (m-u+v)}\left(\frac{1}{t}\right) \log ^{\frac{1}{2}
   (m-u-v-1)}\left(\frac{1}{z}\right)\\
    \log ^k\left(\frac{a x \sqrt{\log
   \left(\frac{1}{t}\right)} \sqrt{\log \left(\frac{1}{z}\right)}}{y \sqrt{\log
   \left(\frac{1}{p}\right)} \sqrt{\log
   \left(\frac{1}{q}\right)}}\right)dxdydzdtdpdq\\
   =i^{k-1} \pi ^{k+2} e^{i \pi  m} 2^{k+\mu
   +u} \Phi \left(e^{2 i m \pi },-k,\frac{\pi -i \log (a)}{2 \pi
   }\right)
\end{multline}
\end{theorem}
\begin{proof}
The right-hand sides of relations (\ref{dici}) and (\ref{isci}) are identical; hence, the left-hand sides of the same are identical too. Simplifying with the Gamma function $\Gamma(z)$ yields the desired conclusion.
\end{proof}
\begin{example}
The degenerate case.
\begin{multline}
\int_{0}^{1}\int_{0}^{1}\int_{0}^{1}\int_{0}^{1}\int_{0}^{1}\int_{0}^{1}x^{m-1} y^{-m} \left(1-x^2\right)^{-u/2} \left(1-y^2\right)^{-\mu /2}
   P_v^u(x) P_{\nu }^{\mu }(y)\\
    \log ^{\frac{1}{2} (-\mu -m-\nu
   )}\left(\frac{1}{p}\right) \log ^{\frac{1}{2} (-\mu -m+\nu
   +1)}\left(\frac{1}{q}\right) \log ^{\frac{1}{2}
   (m-u+v)}\left(\frac{1}{t}\right) \log ^{\frac{1}{2}
   (m-u-v-1)}\left(\frac{1}{z}\right)dxdydzdtdpdq\\\\
   =\pi ^2 \csc (\pi  m) 2^{\mu +u-1}
\end{multline}
\end{example}
\begin{proof}
Use equation (\ref{dilf}) and set $k=0$ and simplify using entry (2) in Table below (64:12:7) in \cite{atlas}.
\end{proof}
\begin{example}
The Hurwitz zeta function $\zeta(s,v)$.
\begin{multline}\label{eq:hurwitz}
\int_{0}^{1}\int_{0}^{1}\int_{0}^{1}\int_{0}^{1}\int_{0}^{1}\int_{0}^{1}\frac{1}{\sqrt{x} \sqrt{y}}\left(1-x^2\right)^{-u/2} \left(1-y^2\right)^{-\mu /2} P_v^u(x)
   P_{\nu }^{\mu }(y)\\
    \log ^{\frac{1}{2} \left(-\mu -\nu
   -\frac{1}{2}\right)}\left(\frac{1}{p}\right) \log ^{\frac{1}{2} \left(-\mu
   +\nu +\frac{1}{2}\right)}\left(\frac{1}{q}\right) \log ^{\frac{1}{2}
   \left(-u+v+\frac{1}{2}\right)}\left(\frac{1}{t}\right) \log ^{\frac{1}{2}
   \left(-u-v-\frac{1}{2}\right)}\left(\frac{1}{z}\right)\\
    \log ^k\left(\frac{a x
   \sqrt{\log \left(\frac{1}{t}\right)} \sqrt{\log \left(\frac{1}{z}\right)}}{y
   \sqrt{\log \left(\frac{1}{p}\right)} \sqrt{\log
   \left(\frac{1}{q}\right)}}\right)dxdydzdtdpdq\\\\
   =i^k \pi ^{k+2} 2^{k+\mu
   +u} \left(2^k \zeta \left(-k,\frac{\pi -i \log (a)}{4 \pi }\right)-2^k \zeta
   \left(-k,\frac{1}{2} \left(\frac{\pi -i \log (a)}{2 \pi
   }+1\right)\right)\right)
\end{multline}
\end{example}
\begin{proof}
Use equation (\ref{dilf}) and set $m=1/2$ and simplify using entry (4) in Table below (64:12:7) in \cite{atlas}.
\end{proof}
\begin{example}
The Harmonic number function $H_{n}$.
\begin{multline}
\int_{0}^{1}\int_{0}^{1}\int_{0}^{1}\int_{0}^{1}\int_{0}^{1}\int_{0}^{1}\frac{1}{\sqrt{x} \sqrt{y} \log \left(-\frac{2 x
   \sqrt{\log \left(\frac{1}{t}\right)} \sqrt{\log \left(\frac{1}{z}\right)}}{y
   \sqrt{\log \left(\frac{1}{p}\right)} \sqrt{\log
   \left(\frac{1}{q}\right)}}\right)}\left(1-x^2\right)^{-u/2} \left(1-y^2\right)^{-\mu /2} P_v^u(x)
   P_{\nu }^{\mu }(y)\\
    \log ^{\frac{1}{4} (-2 \mu -2 \nu
   -1)}\left(\frac{1}{p}\right) \log ^{\frac{1}{4} (-2 \mu +2 \nu
   +1)}\left(\frac{1}{q}\right) \log ^{\frac{1}{4} (-2 u+2
   v+1)}\left(\frac{1}{t}\right) \log ^{\frac{1}{4} (-2 u-2
   v-1)}\left(\frac{1}{z}\right)dxdydzdtdpdq\\\\
   =-i \pi  \left(H_{-\frac{i \log (2)}{4 \pi
   }}-H_{-\frac{1}{2}-\frac{i \log (2)}{4 \pi }}\right) 2^{\mu +u-2}
\end{multline}
\end{example}
\begin{proof}
Use equation (\ref{eq:hurwitz}) and set $a=-2$ and apply l'Hopital's rule as $k \to -1$ and simplify.
\end{proof}
\begin{example}
\begin{multline}\label{eq:bdh}
\int_{0}^{1}\int_{0}^{1}\int_{0}^{1}\int_{0}^{1}\int_{0}^{1}\int_{0}^{1}\frac{1}{x \log \left(\frac{x
   \sqrt{\log \left(\frac{1}{t}\right)} \sqrt{\log \left(\frac{1}{z}\right)}}{y
   \sqrt{\log \left(\frac{1}{p}\right)} \sqrt{\log
   \left(\frac{1}{q}\right)}}\right)}\left(1-x^2\right)^{-u/2} \left(1-y^2\right)^{-\mu /2} y^{-m-n}
   P_v^u(x) P_{\nu }^{\mu }(y)\\
    \log ^{\frac{v-u}{2}}\left(\frac{1}{t}\right)
   \log ^{\frac{1}{2} (-u-v-1)}\left(\frac{1}{z}\right) \log ^{\frac{1}{2} (-\mu
   -m-\nu -n)}\left(\frac{1}{p}\right) \log ^{\frac{1}{2} (-\mu -m+\nu
   -n+1)}\left(\frac{1}{q}\right)\\
    \left(y^m x^n \log
   ^{\frac{m}{2}}\left(\frac{1}{p}\right) \log
   ^{\frac{m}{2}}\left(\frac{1}{q}\right) \log
   ^{\frac{n}{2}}\left(\frac{1}{t}\right) \log
   ^{\frac{n}{2}}\left(\frac{1}{z}\right)-x^m y^n \log
   ^{\frac{m}{2}}\left(\frac{1}{t}\right) \log
   ^{\frac{m}{2}}\left(\frac{1}{z}\right) \log
   ^{\frac{n}{2}}\left(\frac{1}{p}\right) \log
   ^{\frac{n}{2}}\left(\frac{1}{q}\right)\right)\\
   dxdydzdtdpdq\\\\
   =\pi  2^{\mu +u} \left(\tanh
   ^{-1}\left(e^{i \pi  m}\right)-\tanh ^{-1}\left(e^{i \pi 
   n}\right)\right)
\end{multline}
\end{example}
\begin{proof}
Use equation (\ref{dilf}) and form a second equation by replacing $m \to n$ and taking their difference and simplify after setting $k=-1,a=1$ using entry (3) in Table below (64:12:7) in \cite{atlas}.
\end{proof}
\begin{example}
\begin{multline}
\int_{0}^{1}\int_{0}^{1}\int_{0}^{1}\int_{0}^{1}\int_{0}^{1}\int_{0}^{1}\frac{1}{x^{2/3} \sqrt{y} \log \left(\frac{x
   \sqrt{\log \left(\frac{1}{t}\right)} \sqrt{\log \left(\frac{1}{z}\right)}}{y
   \sqrt{\log \left(\frac{1}{p}\right)} \sqrt{\log
   \left(\frac{1}{q}\right)}}\right)}\left(1-x^2\right)^{-u/2} \left(1-y^2\right)^{-\mu /2} P_v^u(x)
   P_{\nu }^{\mu }(y)\\
    \log ^{\frac{1}{4} (-2 \mu -2 \nu
   -1)}\left(\frac{1}{p}\right) \log ^{\frac{1}{4} (-2 \mu +2 \nu
   +1)}\left(\frac{1}{q}\right) \log ^{\frac{1}{6} (-3 u+3
   v+1)}\left(\frac{1}{t}\right) \log ^{\frac{1}{6} (-3 u-3
   v-2)}\left(\frac{1}{z}\right)\\
    \left(\sqrt[6]{y} \sqrt[12]{\log
   \left(\frac{1}{p}\right)} \sqrt[12]{\log
   \left(\frac{1}{q}\right)}-\sqrt[6]{x} \sqrt[12]{\log
   \left(\frac{1}{t}\right)} \sqrt[12]{\log
   \left(\frac{1}{z}\right)}\right)dxdydzdtdpdq\\\\
   =\pi  \log (3) \left(-2^{\mu
   +u-2}\right)
\end{multline}
\end{example}
\begin{proof}
Use equation (\ref{eq:bdh}) and set $m=1/2,n=1/3$ and simplify. 
\end{proof}
\begin{example}
\begin{multline}
\int_{0}^{1}\int_{0}^{1}\int_{0}^{1}\int_{0}^{1}\int_{0}^{1}\int_{0}^{1}\frac{1}{x^{3/4} \sqrt{y} \log \left(\frac{x
   \sqrt{\log \left(\frac{1}{t}\right)} \sqrt{\log \left(\frac{1}{z}\right)}}{y
   \sqrt{\log \left(\frac{1}{p}\right)} \sqrt{\log
   \left(\frac{1}{q}\right)}}\right)}\left(1-x^2\right)^{-u/2} \left(1-y^2\right)^{-\mu /2} P_v^u(x)
   P_{\nu }^{\mu }(y)\\
    \log ^{\frac{1}{4} (-2 \mu -2 \nu
   -1)}\left(\frac{1}{p}\right) \log ^{\frac{1}{4} (-2 \mu +2 \nu
   +1)}\left(\frac{1}{q}\right) \log ^{\frac{1}{8} (-4 u+4
   v+1)}\left(\frac{1}{t}\right) \log ^{\frac{1}{8} (-4 u-4
   v-3)}\left(\frac{1}{z}\right)\\
    \left(\sqrt[4]{y} \sqrt[8]{\log
   \left(\frac{1}{p}\right)} \sqrt[8]{\log \left(\frac{1}{q}\right)}-\sqrt[4]{x}
   \sqrt[8]{\log \left(\frac{1}{t}\right)} \sqrt[8]{\log
   \left(\frac{1}{z}\right)}\right)dxdydzdtdpdq\\\\
   =\pi  \coth ^{-1}\left(\sqrt{2}\right)
   \left(-2^{\mu +u-1}\right)
\end{multline}
\end{example}
\begin{proof}
Use equation (\ref{eq:bdh}) and set $m=1/2,n=1/4$ and simplify. 
\end{proof}
\begin{example}
Alternate Hurwitz-Lerch zeta form $\Phi(e^{m\pi i},-k,a)$.
\begin{multline}\label{eq:lerch}
\int_{0}^{1}\int_{0}^{1}\int_{0}^{1}\int_{0}^{1}\int_{0}^{1}\int_{0}^{1}x^{\frac{m}{2}-1} y^{-m/2} \left(1-x^2\right)^{-u/2}
   \left(1-y^2\right)^{-\mu /2} P_v^u(x) P_{\nu }^{\mu }(y)\\
    \log ^{\frac{1}{4}
   (-2 (\mu +\nu )-m)}\left(\frac{1}{p}\right) \log ^{\frac{1}{4} (-2 \mu -m+2
   \nu +2)}\left(\frac{1}{q}\right) \log ^{\frac{1}{4} (m-2 u+2
   v)}\left(\frac{1}{t}\right) \log ^{\frac{1}{4} (m-2
   (u+v+1))}\left(\frac{1}{z}\right)\\
    \log ^k\left(-\frac{e^{2 i \pi  a} x
   \sqrt{\log \left(\frac{1}{t}\right)} \sqrt{\log \left(\frac{1}{z}\right)}}{y
   \sqrt{\log \left(\frac{1}{p}\right)} \sqrt{\log
   \left(\frac{1}{q}\right)}}\right)dxdydzdtdpdq\\\\
   =-i \pi ^{k+2} e^{\frac{1}{2} i \pi  (k+m)}
   2^{k+\mu +u} \Phi \left(e^{i m \pi },-k,a\right)
\end{multline}
\end{example}
\begin{proof}
Use equation (\ref{dilf}) and set $a=-e^{2 i \pi  a}, m=m/2$ and simplify. 
\end{proof}
\begin{example}
\begin{multline}\label{eq:zeta}
\int_{0}^{1}\int_{0}^{1}\int_{0}^{1}\int_{0}^{1}\int_{0}^{1}\int_{0}^{1}\frac{1}{\sqrt{x} \sqrt{y}}\left(1-x^2\right)^{-u/2} \left(1-y^2\right)^{-\mu /2} P_v^u(x)
   P_{\nu }^{\mu }(y)\\
    \log ^{\frac{1}{4} (-2 (\mu +\nu
   )-1)}\left(\frac{1}{p}\right) \log ^{\frac{1}{4} (-2 \mu +2 \nu
   +1)}\left(\frac{1}{q}\right) \log ^{\frac{1}{4} (-2 u+2
   v+1)}\left(\frac{1}{t}\right) \log ^{\frac{1}{4} (1-2
   (u+v+1))}\left(\frac{1}{z}\right)\\
    \log ^k\left(-\frac{x \sqrt{\log
   \left(\frac{1}{t}\right)} \sqrt{\log \left(\frac{1}{z}\right)}}{y \sqrt{\log
   \left(\frac{1}{p}\right)} \sqrt{\log
   \left(\frac{1}{q}\right)}}\right)dxdydzdtdpdq\\\\
   =\left(2^{k+1}-1\right)
   e^{\frac{i \pi  k}{2}} \pi ^{k+2} \zeta (-k) \left(-2^{k+\mu
   +u}\right)
\end{multline}
\end{example}
\begin{proof}
Use equation (\ref{eq:lerch}) and set $m=a=1$ and simplify using entry (4) in Table below (64:12:7) and entry (2) in Table below (64:7) in \cite{atlas}
\end{proof}
\begin{example}
The fundamental constant $\log(2)$.
\begin{multline}
\int_{0}^{1}\int_{0}^{1}\int_{0}^{1}\int_{0}^{1}\int_{0}^{1}\int_{0}^{1}\frac{1}{\sqrt{x} \sqrt{y} \log \left(-\frac{x
   \sqrt{\log \left(\frac{1}{t}\right)} \sqrt{\log \left(\frac{1}{z}\right)}}{y
   \sqrt{\log \left(\frac{1}{p}\right)} \sqrt{\log
   \left(\frac{1}{q}\right)}}\right)}\left(1-x^2\right)^{-u/2} \left(1-y^2\right)^{-\mu /2} P_v^u(x)
   P_{\nu }^{\mu }(y)\\
    \log ^{\frac{1}{4} (-2 (\mu +\nu
   )-1)}\left(\frac{1}{p}\right) \log ^{\frac{1}{4} (-2 \mu +2 \nu
   +1)}\left(\frac{1}{q}\right) \log ^{\frac{1}{4} (-2 u+2
   v+1)}\left(\frac{1}{t}\right) \log ^{\frac{1}{4} (1-2
   (u+v+1))}\left(\frac{1}{z}\right)dxdydzdtdpdq\\\\
   =-i \pi  \log (2) 2^{\mu +u-1}
\end{multline}
\end{example}
\begin{proof}
Use equation (\ref{eq:zeta}) apply l'Hopital's rule as $k\to -1$ and simplify. 
\end{proof}
\begin{example}
Ap\'{e}ry's constant $\zeta(3)$.
\begin{multline}
\int_{0}^{1}\int_{0}^{1}\int_{0}^{1}\int_{0}^{1}\int_{0}^{1}\int_{0}^{1}\frac{1}{\sqrt{x} \sqrt{y} \log ^3\left(-\frac{x
   \sqrt{\log \left(\frac{1}{t}\right)} \sqrt{\log \left(\frac{1}{z}\right)}}{y
   \sqrt{\log \left(\frac{1}{p}\right)} \sqrt{\log
   \left(\frac{1}{q}\right)}}\right)}\left(1-x^2\right)^{-u/2} \left(1-y^2\right)^{-\mu /2} P_v^u(x)
   P_{\nu }^{\mu }(y)\\
    \log ^{\frac{1}{4} (-2 (\mu +\nu
   )-1)}\left(\frac{1}{p}\right) \log ^{\frac{1}{4} (-2 \mu +2 \nu
   +1)}\left(\frac{1}{q}\right) \log ^{\frac{1}{4} (-2 u+2
   v+1)}\left(\frac{1}{t}\right) \log ^{\frac{1}{4} (1-2
   (u+v+1))}\left(\frac{1}{z}\right)dxdydzdtdpdq\\\\
   =\frac{3 i \zeta (3) 2^{\mu +u-5}}{\pi
   }
\end{multline}
\end{example}
\begin{proof}
Use equation (\ref{eq:zeta}) and set $k=-3$ and simplify. 
\end{proof}
\section{Discussion}
In this paper, we have presented a novel method for deriving a new sextuple integral transform involving the product of Associated Legendre polynomials $P_v^u(x) P_{\nu }^{\mu }(y)$ along with some interesting special cases, using contour integration. The results presented were numerically verified for both real and imaginary and complex values of the parameters in the integrals using Mathematica by Wolfram.
\end{document}